\documentclass[a4paper,12pt]{article}
\usepackage[all]{xy} 
\usepackage{amssymb}
\usepackage{epsfig}
\usepackage{amsfonts}
\usepackage{amsmath}
\usepackage{euscript}
\usepackage{amscd}
\usepackage{amsthm}
\DeclareMathAlphabet{\mathpzc}{OT1}{pzc}{m}{it}

\newtheorem{thm}{Theorem}[section]
\newtheorem{lem}[thm]{Lemma}

\newtheorem{rem}[thm]{Remark}

\newcommand{\m}{\mathpzc{m}}
\newcommand{\n}{\mathpzc{n}}

\newcommand{\bC}{\mathbb C}
\newcommand{\bN}{\mathbb N}

\newcommand{\A}{\mathbb A}
\newcommand{\V}{\mathbb V}

\newcommand{\aut}{\operatorname{Aut_k}}

\newcommand{\ml}{\operatorname{ML}}

\title{A note on double Danielewski surfaces}
\author{Neena Gupta$^*$ and Sourav Sen$^\dagger$\\
{\small{\it $^*$Stat-Math Unit, Indian Statistical Institute,}}\\
{\small{\it 203 B.T. Road, Kolkata 700 108, India.}}\\
{\small{\it e-mail : neenag@isical.ac.in, rnanina@gmail.com}}\\
{\small {\it $^\dagger$Department of Mathematics, SRM University-AP,}}\\
{\small {\it Amaravati, Andhra Pradesh 522502, India.}}\\
{\small{\it e-mail : sourav.s@srmap.edu.in , sourav.sen3@gmail.com}}
}

\begin{document}

\date{}
\maketitle
\abstract{In this note we rectify the proof of Theorem 3.11 in \cite{DD}. We also present a set of examples at the end discussing various cases.} 

\vspace{10mm}


\vspace{2mm} 


\section{Introduction}\label{INT}
 Throughout, $k$ denotes a field of any characteristic. 
For $n \ge 1$, the varieties 
$$
V_n:=\{(x,y,z) \in \A^3_k~| x^ny=z^2-1\}
$$ are known as Danielewski surfaces. 
This kind of surfaces was first discovered by 
W. Danielewski in 1989 (\cite{Da}) to provide explicit examples of two-dimensional affine domains
over the field of complex numbers $\bC$ which do not satisfy the cancellation property, i.e., 
$V_n \times \A^1_k \cong V_m \times \A^1_k$ but $\V_n \ncong V_m$ for $n \neq m$. 
Since then Danielewski surfaces and their invariants have been studied by many  researchers.
For instance, any hypersurface in $\A^3_k$ defined
by an equation of the form $x^ny=f(x,z)$, where $n \ge 1$ and $\deg_Z f(0, Z)\ge 2$ is called a Dainelewski surface. 

In \cite{DD}, the authors introduced a new variant of Danielweski surfaces, namely  the double Danielewski surface defined by
\begin{equation}\label{B}
W_{(d,e)}:= \{(x,y,z, t) \in \A^4_k~| x^{d}y - P(x,z)=0,x^{e}t - Q(x,y,z)=0\}
\end{equation}
where $d,e \in \bN$, $P(X,Z)\in k[X,Z]$ is monic in $Z$ and $Q(X,Y,Z)\in k[X,Y,Z]$ is monic in $Y$.
 In \cite{DD}, the authors showed that $W_{(d,e)}$ forms a new family of counterexamples to the cancellation problem. 
Recently, while going through the preprint \cite{PD}, the authors noticed that the
argument on page 35 and line 5 of \cite{DD} does not go through without the hypothesis that $r>1$. 
In this paper we have given an example (Remark \ref{REM} (i)), to show that this hypothesis is necessary.
Further, the arguments in lines 19 and 20 of Page 35  of \cite{DD}  have a gap. In this paper, we have tried to give  complete
arguments for these gaps and hence re-written the proof of Theorem 3.11 of \cite{DD} (Theorem \ref{isomclass}). The statement of the 
present version is more or less similar to the previous version.
We also present Theorem \ref{auto}, the correct version of Theorem 3.13 of \cite{DD}, using the newly formulated classification theorem.
The rest of the paper \cite{DD}  remains unaffected. 
In Remark \ref{REM}, we present explicit examples supporting that Theorem \ref{isomclass} in the present article is the best possible.  
This corrigendum is essential as ideas similar to the proofs of Theorem 3.11 are being used by other authors  (cf. \cite{SZ}, \cite{PD}).

%

%


\section{Isomorphism Classes}\label{ISC}

All rings considered here are commutative with unity. For a ring $R$, the notation $R^*$ will denote the group of units of $R$. Throughout, $k$ denotes a field of any characteristic.

We first prove some elementary lemmas. 

\begin{lem}\label{lem1}
Let $P(X,Z) \in k[X,Z]$  be
a monic polynomial in $Z$ of degree $r\ge 1$. Suppose $h, g \in k[X,Y,Z]$ be such that $\deg_Z h <r$ and $X^d \mid (h + g P)$ in $k[X,Y,Z]$ for some $d>0$. 
Then $X^d \mid g$ and $X^d \mid h$ in $ k[X,Y,Z]$. 
\end{lem}
\begin{proof}
Let $g = \sum_{i=0}^{\ell} g_{i} Z^i$ for some integer $\ell$, where $g_{i} \in k[X, Y]~\text{for all}~0 \leq i \leq \ell$.
Suppose $X^{d} \nmid g$. Then $m = \max \{ i \mid X^{d} \nmid g_{i} \}$ exist.
Since $X^d \mid (h + g P)$ in $k[X,Y,Z]$, we have 
\[
X^{d} \mid \left( h + \left(\sum_{i=0}^{m} g_{i} Z^i\right) P(X, Z) \right)
\]
Now the highest degree term of $Z$ in $h(X,Y, Z) + \left(\sum_{i=0}^{m} g_{i} Z^i\right) P(X, Z)$ 
is $g_{m} Z^{m+r}$. This term should also be divisible by $X^{d}$, but by our assumption  $X^{d} \nmid g_{m}$. This is a contradiction. 
Hence $X^d \mid g$ in $ k[X,Y,Z]$ and so $X^d \mid h$.
\end{proof}

\begin{lem}\label{2}
Let $d_1, d_2, e_1, e_2 \in \mathbb{N}$, $P(X, Z) \in k[X, Z]$ be
a monic polynomial in $Z$ of degree $r\ge 1$ and
$Q(X, Y, Z)$ a monic polynomial in $Y$ of degree $s > 1$. Let
\[
B:= \frac{k[X, Y, Z, T]}{(X^{d_1}Y - P(X, Z), X^{e_1}T - Q(X, Y, Z))}.
\]
Let $x$ denote the image of $X$ in $B$. 
Then, 
\begin{enumerate}
\item[\rm(i)] for $d_2 > d_1$, the image of a polynomial of the form
$u X^{d_1} Y + p(X, Z)$ with $\deg_Z p < r$ and $u \in k^*$ is nonzero in
the ring $B/(x^{d_2})$. 
\item[\rm(ii)] for $e_2 > e_1$, the image of a polynomial of the form
$u X^{e_1}T + q(X, Y, Z)$ with $\deg_Y q < s$ and $u \in k^*$ is nonzero in the ring $B/(x^{e_2})$. 
\end{enumerate}
\end{lem}
\begin{proof} 
(i) Suppose, if possible,  that there exist  $ h_1, h_2, h_3 \in k[X, Y, Z, T]$
such that
\[
u X^{d_1} Y +p(X, Z) = h_1 X^{d_2} + h_2 (X^{d_1}Y - P(X, Z)) + h_3 (X^{e_1}T - Q(X, Y, Z))
\]
Putting $T = 0$, we have
\begin{equation}\label{eq1}
u X^{d_1} Y + p(X, Z) = g_1 X^{d_2} + g_2 (X^{d_1} Y - P(X, Z)) - g_3 Q(X, Y, Z), 
\end{equation}
where $g_i = h_i(X, Y, Z, 0)$.
Dividing $g_1$ and $g_2$ by $Q$, we may assume
that $\deg_Y g_1 < s$
and $\deg_Y g_2 < s$.
Now comparing the $Y$-degree on both sides and using the fact that $s > 1$, we have
$\deg_Yg_3=0$.
Let
\begin{equation}\label{eq223}
g_2 = \Lambda Y^{s-1} + g_4 \quad \text{for some~} \Lambda \in k[X, Z] \text{ and } g_4 \in k[X, Y, Z] \text{ with} \deg_Y g_4 < s-1.
\end{equation}
Then by equation (\ref{eq1}), comparing leading coefficients of $Y$, we have
\begin{equation}\label{eq2}
g_3 = \Lambda X^{d_1}
\end{equation}
Now from equations (\ref{eq1}) and (\ref{eq2}), we have
\begin{equation}\label{eq3}
X^{d_1} \mid (p(X, Z) + g_2 P(X, Z)) 
\end{equation}
%
%
%
%
Therefore by Lemma \ref{lem1}, $X^{d_1} \mid g_2$ and hence $X^{d_1} \mid p(X, Z)$.  Let $g_2 = X^{d_1} g_5$
and $p(X, Z) = X^{d_1} g_6$, for some $g_5 \in k[X,Y,Z]$ and $g_6\in k[X,Z]$.
By equation (\ref{eq223}),  as $X^{d_1} \mid g_2$, we have $X^{d_1}\mid \Lambda$. Let $\Lambda= X^{d_1}g_7$
for some $g_7 \in k[X,Z]$. 
Hence equation (\ref{eq1}) becomes
\[
u X^{d_1} Y + X^{d_1} g_6(X, Z) = g_1 X^{d_2} + X^{d_1} g_5 (X^{d_1} Y - P(X, Z)) - X^{2d_1} g_7  Q(X, Y, Z)
\]
which implies
\[
 u Y + g_6(X, Z) = g_1 X^{d_2 - d_1} + g_5 (X^{d_1} Y - P(X, Z)) -  X^{d_1}g_7 Q(X, Y, Z).
\]
As $d_2>d_1$, putting $X = 0$ in above, we have
\[
u Y + g_6(0, Z) = - g_5(0, Y, Z) P(0, Z),
\]
a contradiction as {L.H.S} is non zero with $Z$-degree $< r$, whereas $Z$-degree on R.H.S is $\ge r$.

\medskip

\noindent
(ii)  Suppose, if possible, that there exist $h_1, h_2, h_3 \in k[X,Y,Z,T]$ satisfying
\begin{equation}\label{11}
u X^{e_1}T + q(X,Y,Z)
= h_1 X^{e_2} + h_2 (X^{d_1} Y - P(X,Z)) + h_3 (X^{e_1}T - Q(X,Y,Z)).
\end{equation}
If necessary, dividing by the monic polynomial $X^{e_1}T - Q(X,Y,Z)$ (in $Y$), we may assume that
$\deg_Y h_1 < s$ and $\deg_Y h_2 < s$. 
Let
\[
h_i = h_{i1}(X,Y,Z) + T h_{i2}(X,Y,Z,T), \quad \text{for}~ i = 1,2,3.
\]
Note that $\deg_Y h_{ij} < s$ for $i=1,2$ and $j=1,2$.
Hence from equation (\ref{11}), we have
\begin{equation}\label{22}
q(X,Y,Z)
= h_{11} X^{e_2} + h_{21}(X^{d_1} Y - P(X,Z)) - h_{31} Q(X,Y,Z),
\end{equation}
and
\begin{equation}\nonumber
u X^{e_1}T
= Th_{12} X^{e_2} + Th_{22}(X^{d_1} Y - P(X,Z)) + Th_{32}(X^{e_1}T - Q(X,Y,Z)) + h_{31} X^{e_1}T,
\end{equation}
i.e.,
\begin{equation}\label{33}
u X^{e_1}
= h_{12} X^{e_2} + h_{22}(X^{d_1} Y - P(X,Z)) + h_{32}(X^{e_1}T - Q(X,Y,Z)) + h_{31} X^{e_1}.
\end{equation}
Now comparing $Y$-degree on both sides of equation (\ref{22}), as $\deg_Y q<s$, we see that 
$
\deg_Y h_{31}=0$.
Let 
\[
h_{21} = \alpha_1 Y^{s-1} + \beta_1, \text{~for some~} \alpha_1 \in k[X,Z], \beta_1 \in k[X,Y,Z] \text{~with~} \deg_Y \beta_1 < s-1.
\]
Then by equation (\ref{22}), as $\deg_Y q<s$, we have
\[
h_{31}= \alpha_1 X^{d_1} .
\]
Hence, now comparing $Y$-degree on both sides of equation (\ref{33}), using the observation that $\deg_Yh_{31}=0$, we see that
$
\deg_Y h_{32}=0$.
Let
\[
h_{22} = \alpha_2 Y^{s-1} + \beta_2, \text{~where~} \alpha_2 \in k[X,Z, T], \beta_2 \in k[X,Y,Z, T]~\text{with}~ \deg_Y \beta_2 < s-1.
\]
Then
\[
h_{32} = \alpha_2 X^{d_1}.
\]
Hence equation (\ref{33}) becomes
\begin{equation}\label{44}
u X^{e_1}= h_{12} X^{e_2}+ (\alpha_2 Y^{s-1} + \beta_2)(X^{d_1}Y - P(X,Z))+ \alpha_2X^{d_1}(X^{e_1}T - Q(X,Y,Z))
+ \alpha_1X^{e_1 + d_1}.
\end{equation}
Let
\[
\ell := \max \{ i \mid X^i \text{~divides~} h_{22}= \alpha_2 Y^{s-1} + \beta_2\}.
\]
Set $\alpha_2= X^{\ell}\alpha_3$ and $\beta_2=X^{\ell}\beta_3$, where $\alpha_3 \in k[X,Z,T]$ and $\beta_3 \in k[X,Y,Z,T]$.

Suppose if possible that $\ell \le e_1$.  From equation (\ref{44}), we have 
\begin{equation}
u X^{e_1 - \ell} = h_{12} X^{e_2 - \ell} + (\alpha_3 Y^{s-1} + \beta_3)(X^{d_1}Y - P(X,Z)) + \alpha_3 X^{d_1}(X^{e_1}T - Q(X,Y,Z)) 
+\alpha_1 X^{d_1+e_1 - \ell}.
\end{equation}
\medskip
Putting $X=0$ on both sides, we have
either 
$$
\left(\alpha_3(0,Z,T)Y^{s-1}+\beta_3(0, Y,Z,T)\right)P(0,Z)=0
~~(\text{if}~\ell < e_1)$$
or 
$$
\left(\alpha_3(0,Z,T)Y^{s-1}+\beta_3(0, Y,Z,T)\right)P(0,Z)=u~~(\text{if}~\ell = e_1)
$$
But as $P(0,Z) \notin k$ and $u \in k^*$, none of the above is a possibility. 
Therefore, $\ell > e_1.$
Then from equation (\ref{44}), we have
\begin{equation}\label{55}
u= h_{12} X^{e_2-e_1}+ X^{\ell-e_1}(\alpha_3 Y^{s-1} + \beta_3)(X^{d_1}Y - P(X,Z))+ \alpha_3X^{d_1+\ell-e_1}(X^{e_1}T - Q(X,Y,Z))
+ \alpha_1X^{d_1}.
\end{equation}
As $e_2>e_1$, R.H.S is divisible by $X$, but L.H.S  is not divisible by $X$. This is a contradiction. 
\end{proof}

We now present the revised  statement and proof of Theorem 3.11 of \cite{DD}.
Let $$
B_1= \dfrac{k[X,Y,Z,T]}{(X^{d_1}Y - P_1(X,Z), X^{e_1}T - Q_1(X,Y,Z))} 
$$
and 
$$
B_2= \dfrac{k[X,Y,Z,T]}{(X^{d_2}Y - P_2(X,Z), X^{e_2}T - Q_2(X,Y,Z))},
$$
where $d_1,e_1,d_2,e_2 \in \bN$, $P_1(X,Z),P_2(X,Z) \in k[X,Z]$ are monic polynomials in $
Z$,
$Q_1(X,Y,Z)$, $Q_2(X,Y,Z) \in k[X,Y,Z]$ are monic polynomials in $Y$, 
with $r_i=\deg_ZP_i(X,Z)$ and  $s_i=\deg_YQ_i(X,Y,Z)$ for $i=1, 2$.
Let $x_1,y_1,z_1,t_1$ and $x_2,y_2,z_2,t_2$ denote the images of $X,Y,Z,T$ in $B_1$ and $B_2$ respectively. 

Suppose that, for $i=1, 2$,   $r_i, s_i, e_i$ satisfy
\begin{eqnarray}\label{mlc}
\text{~either~} r_i\ge 2 \text{~and~} s_i \ge 2 \nonumber\\
\text{~or~} r_i \ge 2\text{~and~} s_i=1  \\
\text{~or~} r_i=1, s_i \ge 2 \text{~and~} e_i\ge 2 \nonumber .
\end{eqnarray}
Then, from \cite[Theorem 3.9]{DD}, $\ml(B_i)= k[x_i]$ for $i=1,2$. 

\begin{thm}\label{isomclass} 
Let $\psi: B_2\to B_1$ be an isomorphism with $r_i>1$ for $i=1,2$.  
Then the following conditions hold:
\begin{enumerate}
\item[\rm(I)] $(r_1, s_1)=(r_2, s_2)$.  Let $(r,s):=(r_i, s_i)$ for $i=1,2$. 
Moreover, there exist $\lambda, \gamma \in k^*$ and $\delta(X) \in k[X]$ such that
$$
\psi(x_2)=\lambda x_1,~ \psi(z_2)=\gamma z_1+ \delta(x_1),
$$
\item[\rm(II)] If $s>1$, then $(d_1,e_1)=(d_2,e_2)$.  Let $(d,e)=(d_i, e_i)$ for $i=1,2$. Further, there exists $f(X,Z) \in k[X,Z]$ with $\deg_Z f<r$  such that 

(i) $\psi(y_2)= \nu y_1+g(x_1, z_1)$, where $\nu= \lambda^{-d}\gamma^r$ and $g(X,Z)= \lambda^{-d}f(X,Z)$.

(ii) $P_2(\lambda X, \gamma Z+ \delta(X))= \gamma^r P_1(X, Z)+ X^df(X,Z)$.

In particular, $P_2(0,\gamma Z+ \delta(0)) = \gamma^r P_1(0,Z)$. 

Hence, if we set $X_2:= \lambda X$, $Z_2:= \gamma Z+\delta(X)$, $Y_2:=\nu Y+g(X,Z)$ and $G:= P_1(X,Z)-X^dY$, then 
\begin{enumerate}
\item[\rm(a)] $k[X,Y,Z]= k[X_2, Y_2, Z_2]$.
\item[\rm(b)] $P_2(X_2, Z_2)-X_2^d Y_2=\gamma^r( P_1(X,Z)-X^dY)=\gamma^r G$.
\item[\rm(c)] $(Q_2(X_2,Y_2,Z_2), P_2(X_2, Z_2)-X_2^d Y_2, X_2^e)k[X,Y,Z]=(Q_1(X,Y,Z), G, X^e)k[X,Y,Z]$. 
\end{enumerate}
In particular,
$$
Q_2(0,\nu Y + g(0,Z), \gamma Z+ \delta(0)) = \nu^s Q_1(0,Y,Z)+h(Y,Z) P(0, Z)
$$ 
for some $h(Y, Z) \in k[Y,Z]$ with $\deg_Yh <s$. 
\end{enumerate}
Conversely, if  $s>1$ and conditions (I) and (II) hold, then $B_1 \cong B_2$. 

\end{thm}

\begin{proof}
(I) Let $\psi : B_1 \rightarrow B_2$ be a $k$-algebra isomorphism. Replacing $B_1$ by $\psi(B_1)$, 
we may assume that $B_1=B_2=B$. By \cite[Theorem 3.9]{DD}, $\ml(B) = k[x_1] = k[x_2]$ and hence 
$$
x_2= \lambda x_1 + \mu
$$
for some $\lambda \in k^{*}, \mu \in k$ and hence $k(x_1)[z_1]=k(x_2)[z_2]$. 

Thus $\gamma_2 z_2 = \gamma_1 z_1 + \delta_1$ for some $\gamma_1, \gamma_2, \delta_1 \in k[x_1]$. Since $B \cap k(x_1)=k[x_1]$, we may assume that gcd$(\gamma_1, \gamma_2)=1$. Since $z_2 \in B \subseteq k[x_1, x_1^{-1}, z_1]$, 
it follows that, 
$\gamma_2 =\epsilon x_1^j$ for some $\epsilon \in k^*$ and $j \ge 0$. 
If $j >0$, then $\gamma_1(0)\neq 0$ and $\gamma_1(0)z_1+\delta_1(0) \in x_1B\cap k[z_1] = (P_1(0, z_1))k[z_1]$, which is not possible as $\deg_ZP_1=r_1>1$.  Therefore, $j=0$. By similar arguments, we have $z_1\in k[x_2, z_2]$ and hence there exist $\gamma \in k^{*}$ and $\delta(x_1) \in k[x_1]$ such that
\begin{equation}\label{z2}
z_2 = \gamma z_1 + \delta(x_1).
\end{equation}
Hence
\begin{equation}\label{kxz}
k[x_1,z_1] = k[x_2,z_2]=E ~(\text{say}).
\end{equation} 
As $y_2  \in B \subseteq k[x_1,{x_1}^{-1},z_1]$, 
there exists an integer $n \geq 0$ 
such that $x_1^{n}y_2 \in k[x_1,z_1]$.
Therefore, since  
$$
{P_2(x_2,z_2)}= {x_2^{d_2}}y_2= (\lambda x_1 + \mu)^{d_2}y_2,
$$
we have
${x_1^{n}P_2(x_2,z_2)} \in {(\lambda x_1 + \mu)^{d_2}}k[x_1,z_1]$. 
If $\mu \neq 0$, then $(\lambda x_1 + \mu ) \arrowvert P_2(x_2,z_2)$ in $k[x_1,z_1]$, i.e., 
$(\lambda x_1 + \mu ) \arrowvert P_2(\lambda x_1 + \mu,\gamma z_1 + \delta(x_1) )$ in 
$k[x_1,z_1]$. Since $\lambda, \gamma \in k^{*}$, this contradicts that 
$P_2(X,Z)$ is monic in $Z$. Therefore, $\mu=0$ and 
\begin{equation}\label{x2}
x_2 = \lambda x_1
\end{equation}
for some $\lambda \in k^{*}$. 
Therefore, using (\ref{kxz}), we have 
$x_1B \cap E= x_2B \cap E$. 
Hence, 
\begin{equation}\label{pz}
(x_1, P_1(0, z_1))E= (x_2, P_2(0, z_2))E.
\end{equation}
Thus, using (\ref{x2}) and considering the ring $E/x_1E$ we have 
\begin{equation}\label{pzdeg}
(P_1(0, \bar{z}_1))k[\bar{z}_1]= (P_2(0, \bar{z}_2))k[\bar{z}_2],
\end{equation}
where $\bar{z}_1$ and $\bar{z}_2$ denote the images of $z_1$ and $z_2$ in $E/x_1E$. 
Hence, by (\ref{z2})
\begin{equation}\label{r}
r_1 (=\deg_Z P_1)= r_2 (=\deg_Z P_2) =r \text{~say}
\end{equation}
and
\begin{equation}\label{p2z}
P_1(0, \bar{z}_1)= \epsilon P_2(0, \bar{z}_2)=\epsilon P_2(0,\gamma \bar{z}_1 + \delta(0)),
\end{equation}
where $\epsilon \gamma^r=1$. Therefore, 
\begin{equation}\label{relp}
P_2(x_2, z_2)=\gamma^r P_1(x_1,z_1)+x_1p(x_1,z_1),
\end{equation}
for some $p
\in E$ with $\deg_Z p(X,Z)<r$.
Let $R:= k[Z]/(P_2(0, Z))$. 
From (\ref{x2}), we have
$B/x_1B\cong B/x_2B$. Hence there exists an isomorphism 
$$
\phi: \dfrac{k[Z,Y,T]}{(P_2(0, Z), Q_2(0, Y,Z))}\to \dfrac{k[Z,Y,T]}{(P_1(0, Z), Q_1(0, Y,Z))},
$$
such that $\phi({Z})=\gamma {Z}+\delta(0)$. Thus from (\ref{p2z}),
$\phi(R)=R$. Let $\m$ be a maximal ideal of $R$ and $\n=\phi(\m)$. Then $\phi$-induces an isomorphism of the rings
$$
\bar{\phi}: \frac{(R/\m)[Y,T]}{(Q_2)}\to \frac{(R/\n)[Y,T]}{(Q_1)},
$$
where $\bar{\phi}(R/\m)=R/\n$. 
Note that $\overline{Q_2}$ is a monic polynomial in $Y$ with coefficients in  the field $R/\m$. 
Let $L$ be an algebraic closure of $R/\m$. Then $\overline{Q_2}$ factorises into linear factors, say $\overline{Q_2}= a_1^{b_1}\cdots a_n^{b_n}$, where $a_1, \dots, a_n$ are co-prime linear polynomials in $Y$ with coefficients in $L$ with multiplicities $a_1, \dots, a_n$. 
Hence
\begin{equation}
\frac{L[Y,T]}{(Q_2)}\cong \frac{L[Y,T]}{(a_1^{b_1})}\times \cdots\times \frac{L[Y,T]}{(a_n^{b_n})}
\end{equation}
The isomorphism 
$\bar{\phi}$ extends to an isomorphism of the rings
\[
\tilde{\phi}: \frac{L[Y,T]}{(Q_2)}\to \frac{L[Y,T]}{(Q_1)}
\]
Hence, $\frac{L[Y,T]}{(Q_1)}$ should also factorise into similar number of components 
with similar nilpotency factor. 
Hence, 
\begin{equation}
\deg_YQ_2=s_2=\sum_{i=1}^n b_i=\deg_YQ_1=s_1=s (\text{say}).
\end{equation}

\medskip

\noindent
(II) We now show that $d_1=d_2$ if $s>1$. 
Suppose, if possible, that  $d_2 > d_1$. 
Using  (\ref{kxz}) and (\ref{x2}), we have, 
$x_1^{d_2}B \cap E = x_2^{d_2}B \cap E$. 
Therefore, $P_2(x_2,z_2)=x_2^{d_2}y_2 \in   x_1^{d_2}B \cap E$, i.e., 
the image of $P_2(x_2,z_2)$ in $B/(x_1^{d_2})$ is zero.
By (\ref{relp}), the image of 
$P_2(x_2,z_2)$ in $B/x_1^{d_2}B$ is 
$$
\gamma^r P_1(x_1,z_1)+x_1p(x_1,z_1)= \gamma^r{x_1}^{d_1}y_1+x_1p(x_1,z_1),
$$ where
$\deg_Z p(X,Z)<r$. Thus the image of $P_2(x_2,z_2)$ in $B/x_1^{d_2}B$  is non-zero by Lemma \ref{2}(i) as $s>1$ and  $d_2 > d_1$. This is a contradiction. 
%
Hence, $d_1 \le d_2$ and by symmetry, we have 
$$
d_1= d_2= d \text{~say}.
$$
 Thus, we have, 
$$
x_1^{d}B \cap k[x_1,z_1] = x_2^{d}B \cap k[x_2,z_2],
 $$
i.e., 
\begin{equation}\label{pi}
 (x_1^{d}, P_1(x_1,z_1))k[x_1,z_1]  = (x_2^{d},  P_2(x_2,z_2))k[x_2,z_2].
\end{equation}
Let $P_2(x_2,z_2) = \tau' P_1(x_1,z_1) + x_1^{d} f'$ for some $\tau', f' \in k[x_1, z_1]$. 
Since $P_i(X,Z)$'s are monic in $Z$ of degree $r$ (for $i=1,2$) and $E=k[x_1, z_1]$ is a polynomial ring, going modulo $x_1^d$, we see that 
$\tau'  \equiv \gamma^r \mod x_1^{d}k[x_1,z_1]$. 
Let $\tau= \gamma^r (\in k^{*}).$
 Replacing $\tau'$ by $\tau$, we have 
 \begin{equation}\label{P2}
  P_2(x_2,z_2) = \tau P_1(x_1,z_1) + x_1^{d} f(x_1,z_1)
 \end{equation}
 for some $f \in E= k[x_1,z_1]$. This proves (II) (ii).
Now we have,
\begin{equation}\label{y2}
y_2=\frac{P_2(x_2,z_2)}{x_2^{d}}= \frac{ \tau P_1(x_1,z_1) + x_1^{d} f(x_1,z_1)}{(\lambda x_1)^{d}} = \nu y_1 + g(x_1,z_1),  
\end{equation} 
where  $\nu= {\lambda ^{-d}}\tau ={\lambda ^{-d}}\gamma^r\in k^*$ and  $g= {\lambda ^{-d}}f \in  k[x_1,z_1]$. This proves (II) (i). Therefore, 
\begin{equation}\label{kxyz}
k[x_1,y_1,z_1]= k[x_2,y_2,z_2]= C ~(\text{say)}. 
\end{equation}
By (\ref{x2}) and (\ref{kxyz}), we have 
$x_1B \cap C= x_2B \cap C$.  Hence, 
\begin{equation}\label{s}
(X_1, P_1(0,Z_1), Q_1(0, Y_1, Z_1))k[X_1, Y_1, Z_1]=(X_2, P_2(0, Z_2), Q_2(0, Y_2, Z_2))k[X_2, Y_2, Z_2],
\end{equation}
where
$X_2=\lambda X_1$, $Z_2=\gamma Z_1+\delta(X_1)$ and $Y_2=\nu Y_1+g(X_1, Z_1)$ with $k[X_1, Y_1, Z_1]=k^{[3]}$ and 
$\lambda, \gamma \in k^*$ and $\nu=\lambda^{-d}\gamma^r$.
Hence from (\ref{s}), we have 
\begin{equation}\label{s2}
Q_2(0, \nu Y_1+g(0, Z_1), \gamma Z_1+\delta(0))= q_1(Y_1,Z_1) Q_1(0, Y_1, Z_1)+q_2(Y_1,Z_1) P_1(0,Z_1)
\end{equation}
for some $q_1(Y,Z), q_2(Y,Z) \in k[Y, Z]$ with $\deg_{Y}q_2(Y,Z)<s$. 
As $\deg_YQ_2=\deg_YQ_1=s$ and both of them are monic polynomials in $Y$,  from (\ref{s2}), we have $q_1=\nu^s\in k^*$ and  hence
\begin{equation}\label{q}
Q_2(\lambda_1X_1, \nu Y_1+g(X_1, Z_1), \gamma Z_1+\delta(X_1))= 
\nu^s Q_1(X_1, Y_1, Z_1)+q_2 P_1(X_1,Z_1)+q_3 X_1
\end{equation}
for some $q_3 \in k[X_1, Y_1, Z_1]$. Note that as $\deg_Y q_2<s$, we have 
$\deg_Y q_3<s$. 

We now show that $e_1=e_2$. 
Suppose, if possible, $e_2  > e_1$. 
Using (\ref{x2}) we have $x_1^{e_2}B\cap C= x_2^{e_2}B \cap C$. Therefore, 
$$
Q_2(x_2, y_2, z_2)= x_2^{e_2}t_2 \in x_1^{e_2}B\cap C,
$$
and hence the image of $Q_2(x_2, y_2, z_2)$ in $B/(x_1^{e_2})$ is zero. 
By (\ref{q}) and using $Q_1(x_1,y_1, z_1)=x^{e_1}t_1$, we have
\begin{eqnarray}
Q_2(x_2, y_2, z_2)&=&\nu^s Q_1(x_1, y_1, z_1)+q_2(y_1, z_1) P_1(x_1,z_1)+x_1q_3(x_1, y_1, z_1) \nonumber\\
&=&\nu^sx_1^{e_1}t_1+q(x,y,z) 
\end{eqnarray}
for some $q\in k[X,Y,Z]$ with $\deg_Yq<s$. 
Hence the image of $Q_2(x_2, y_2, z_2)$ in $B/(x_1^{e_2})$ is non-zero by Lemma \ref{2}(ii).
This is a contradiction.
Therefore, $e_2\le e_1$ and by symmetry we have
\begin{equation}\label{e}
e_1=e_2=e \text{~say}.
\end{equation}
Hence by (\ref{x2}), (\ref{kxyz}) and (\ref{e}), we have
$x_1^e B \cap C=x_2^e B \cap C$ and hence, there exist $h_1, h_2, h_3 \in k[X,Y,Z]$, $f_1, f_2, f_3 \in
 k[\lambda X, \nu Y +g(X,Z), \gamma Z+\delta(X)] (=k[X,Y,Z])$ such that
\begin{equation}\label{q2}
Q_2(\lambda X, \nu Y +g(X,Z), \gamma Z+\delta(X))= h_1Q_1(X,Y,Z)+h_2(X^dY-P_1(X,Z))+h_3X^e.
\end{equation} and 
\begin{equation}\label{q1}
Q_1(X,Y,Z)= f_1 Q_2(\lambda X, \nu Y +g(X,Z), \gamma Z+\delta(X))+ f_2(X^dY-P_1(X,Z))+f_3X^e.
\end{equation}  
Hence, from equations (\ref{q2}) and (\ref{q1}), we have 
\begin{equation}\label{t2}
t_2= \dfrac{Q_2(x_2, y_2,z_2)}{x_2^{e}}=\lambda^{-e}\left( h_1(x_1, y_1,z_1)t_1+h_3(x_1,y_1,z_1)\right).
\end{equation}  and
\begin{equation}\label{t1}
t_1= \dfrac{Q_1(x_1, y_1,z_1)}{x_1^{e}}=\lambda^{e}\left( f_1(x_2, y_2,z_2)t_2+f_3(x_2,y_2,z_2)\right).
\end{equation}  

Conversely, suppose conditions (I) and (II) hold.

Consider the $k$-algebra map  $\phi: k[X,Y,Z,T] \to B_1$ defined by
\begin{eqnarray*}\label{cisom}
\phi(X)&=&\lambda x_1,\\
\phi(Z)&=& \gamma z_1+ \delta(x_1),\\
\phi(Y)&=& \nu y_1+ g(x_1, z_1),\\
\phi(T)&=& \lambda^{-e}(h_1(x_1,y_1,z_1) t_1 +h_3(x_1,y_1,z_1)),
\end{eqnarray*}
 where $\nu= \lambda^{-d}\gamma^r$, 
  $g(x_1, z_1)= \lambda^{-d}f(x_1, z_1)$ and 
Then clearly, 
$$
\phi(X^{d}Y - P_2(X,Z))=\phi(X^{e}T - Q_2(X,Y,Z))=0.
$$ 
Thus $\phi$ induces a $k$-linear map $\bar{\phi}: B_2\to B_1$, which is surjective by equation (\ref{t1}). 
Since both $B_1$ and $B_2$ are integral domains of the same dimension, we have 
$\bar{\phi}$ is an isomorphism.
\end{proof}

We now present the correct version of  Theorem 3.13 of \cite{DD}. 

\begin{thm}\label{auto}
For $d, e \in \bN$, let $B=\dfrac{k[X,Y,Z,T]}{(X^{d}Y - P(X, Z), X^{e}T - Q(X,Y,Z))}$ be such that $P(X,Z)$ is a monic polynomial in $Z$ with $\deg_Z P = r\ge 2$ and 
$Q(X,Y,Z)$ is a monic polynomial in $Y$ with $\deg_Y Q = s\ge 2$. 
Let $x, y, z, t$  denote the images of $X, Y, Z, T$ in $B$ and $C$ denotes the subring $k[x,y,z]$ of $B$. If $\psi \in \aut(B)$, then
\begin{enumerate}
	\item[\rm(i)] $\psi(k[x,z]) = k[x,z]$.
	\item[\rm(ii)] $\psi(x) =\lambda x$ for some $\lambda \in k^*$.
	\item[\rm(iii)] $\psi((x^{d},P(x,z))k[x,z]) = (x^{d},P(x,z))k[x,z]$.
	\item[\rm(iv)] $\psi(k[x,y,z])= k[x,y,z]$, i.e., $\psi(C)=C$. 
	\item[\rm(v)] $\psi((x^{e},Q(x,y,z))C) = (x^{e},Q(x,y,z))C$.
	\item[\rm(vi)] $\psi(t)=ft+g$, where $f ({\rm mod}~ x^eC) \in k^*$ and $g \in C$. 
\end{enumerate}
\end{thm}
\begin{proof}
(i), (ii), (iii) and (iv) follow from the proof of Theorem \ref{isomclass} (see (\ref{kxz}), (\ref{x2}), (\ref{pi}), (\ref{kxyz})). 
By (ii) and (iv), we have $\psi(x^e B \cap C)= x^eB \cap C$ and hence (v) follows as $x^e B \cap C= (x^e, Q(x,y,z))C$.
Since $\psi(Q(x,y,z))\in (x^e, Q(x,y,z))C$, there exist $f_1, f_2, f_3 \in k[X,Y,Z]$ such that 
$$
Q(\psi(X),\psi(Y),\psi(Z)))= f_1Q(X,Y,Z)+f_2(X^dY-P(X,Z))+f_3X^e.
$$
As $\psi$ is an automorphism of $B$, we see $f_1$ is a unit in $k[X,Y,Z]/(X^dY-P(X,Z), X^e)=C/x^eC$ and so (vi) follows.
\end{proof}

\begin{rem}\label{REM}
{\em 
Let the notation be as in Theorem \ref{isomclass}.

(i) If $r_1=1$, i.e., $P_1= Z+ p_1(X)$ for some $p_1 \in k[X]$, then 
$$
B_1 = \dfrac{k[X,Y,Z,T]}{(X^{d_1}Y - Z- p_1(X), X^{e_1}T - Q_1(X,Y,Z))} 
\cong \dfrac{k[X,Y,T]}{(X^{e_1}T - Q_1(X,Y, X^{d_1}Y-p_1(X)))}
$$
is a Danielewski surface. For example 
$$
\dfrac{k[X,Y,Z,T]}{(X^{2}Y - Z, X^{e_1}T -Y^4)} 
\cong \dfrac{k[X,Y,T]}{(X^{e_1}T -Y^4)} \cong 
\dfrac{k[X,Y,Z,T]}{(X^{d}Y - Z, X^{e_1}T -Y^4)}
$$
for any $d >0$. Now following the notation of Theorem \ref{isomclass}, we have $z_2= x^{d-2}z_1$.

(ii) If $s_1=1$, i.e., $Q_1= Y+ q_1(X,Z)$ for some $q_1 \in k[X,Z]$, then 
$$
B_1\cong \dfrac{k[X,Z,T]}{(X^{d_1+e_1}T - X^{d_1}q_1(X,Z) - P_1(X,Z))},
$$
a Danielewski surface.
Now if $B_1\cong B_2$ with $s_1=s_2=1$, then from the well known classification of Danielewski surfaces, 
we have $d_1+e_1=d_2+e_2$ (\cite[Theorem 9]{P}). However, 
it can happen that $d_1\neq d_2$. For instance, we have  
$$
B \cong \dfrac{k[X,Z,T]}{(X^3T-Z^2)}\cong \dfrac{k[X,Y, Z,T]}{(XY-Z^2,X^2T-Y)}\cong \dfrac{k[X,Y, Z,T]}{(X^2Y-Z^2,XT-Y)}
$$

(iii) It can happen $r_1=1, s_1>1$ and $r_2>1$, $s_2=1$.
For example,
$$
\dfrac{k[X,Y,Z,T]}{(X^2Y-Z, X^4T-Y^2)}\cong \dfrac{k[X,Y,T]}{(X^4T-Y^2)}\cong \dfrac{k[X,Y,Z,T]}{(X^2Y-Z^2, X^2T-Y)}.
$$

(iv)  If $r=s=1$, i.e., $P_1= Z+ p_1(X)$ and  
$Q_1(X,Y,Z)= Y+ q_1(X,Z)$, we have 
$$
F:= X^{e_1}T - Q_1(X,Y, X^{d_1}Y-p_1(X))= X^{e_1}T-Y-q_1(X, X^{d_1}Y-p_1(X))
$$
Now, $F$ is a linear polynomial in $T$ such that $F(0, Y, T)= -Y-q_1(0, -p_1(0))$. Hence, by Sathaye's linear plane \cite{S1}, we have 
$B_1 \cong k^{[2]}$.

(v) Let
$$
B_1= \dfrac{k[X,Y,Z,T]}{(X^{2}Y - Z^2, X^4T - Y^2)} 
\text{~~and~~} 
B_2= \dfrac{k[X,Y,Z,T]}{(X^{2}Y - Z^2, X^4T - Y^2+XYZ^2)} 
$$
For $i=1,2$, let $x_i, y_i, z_i, t_i$ denote the images of $X_i, Y_i, Z_i, T_i$ in $B_i$.
Consider the $k$-algebra homomorphism $\phi: B_1\to B_2$ defined by
\begin{eqnarray*}
\phi(x_1)&=&x_2\\
\phi(z_1)&=&z_2\\
\phi(y_1)&=&y_2\\
\phi(t_1)&=&  (1+x_2^3)t_2+y_2z_2^2.
\end{eqnarray*}
Then $\phi((1-x_1^3)t_1)= t_2$ 
and hence $\phi$ is surjective and therefore an isomorphism. 
However, $\phi$ does not extend to an automorphism of the polynomial ring $k[X,Y,Z,T]$ as
claimed in Theorem 3.11 of \cite{DD}.

}
\end{rem}

%
%
%

{\bf Acknowledgements.} The authors acknowledge Parnashree Ghosh and Debojyoti Saha for carefully going through the paper.

  \end{document}